\documentclass[reqno,10pt]{amsart}
\usepackage{graphicx} 
\usepackage{amsmath}
\usepackage{amsthm}
\usepackage[shortlabels]{enumitem}
\usepackage{amsfonts}
\usepackage{xcolor}
\usepackage{bbm}

\usepackage[backref=page]{hyperref}

\renewcommand*{\backrefalt}[4]{%
\ifcase #1 %
No citations.%
\or
Cited: p. #2.%
\else
Cited: pp. #2.%
\fi
}

\def\[#1\]{\begin{align}#1\end{align}}
\def\(#1\){\begin{align*}#1\end{align*}}

\title[Products of exact systems and Lorentzian CFs]{Products of exact dynamical systems and Lorentzian continued fractions}

\author[B. Barreto]{Brandon G. Barreto-Rosa}
\address{Department of Mathematics\\
George Mason University\\
4400 University Drive, MS: 3F2\\
Fairfax, Virginia 22030}
\email{bbarreto@gmu.edu}

\author[J.-P. Burelle]{Jean-Philippe Burelle}
\address{
Département de Mathématiques\\
Université de Sherbrooke\\
2500 boulevard de l'Université\\
Sherbrooke (QC) Canada, J1K 2R1}
\email{j-p.burelle@usherbrooke.ca}

\author[A. Lukyanenko]{Anton Lukyanenko}
\address{
Department of Mathematics\\
George Mason University\\
4400 University Drive, MS: 3F2\\
Fairfax, Virginia 22030}
\email{alukyane@gmu.edu}

\author[M. Richey]{Martha Richey}
\address{
Temple College\\
516 N Main St\\
Taylor, TX 76574
}
\email{martha.hartt@templejc.edu}


\date{May 28, 2025}

\theoremstyle{definition}

\newtheorem{remark}{Remark}
\newtheorem{theorem}{Theorem}

\newtheorem{lemma}{Lemma}
\newtheorem{corollary}{Corollary}

\newcommand{\id}{\operatorname{id}}
\newcommand{\jj}{\mathbbm{j}}
\newcommand{\dd}{\mathbf{d}}
\newcommand{\bR}{\mathbb{R}}
\newcommand{\N}{\mathbb{N}}

\newcommand{\R}{\mathbb{R}}
\newcommand{\Q}{\mathbb{Q}}
\newcommand{\bZ}{\mathbb{Z}}
\newcommand{\Z}{\mathbb{Z}}
\newcommand{\X}{\mathbb X}
\newcommand{\Sym}{\mathrm{Sym}}
\newcommand{\Zee}{\mathcal Z}
\newcommand{\norm}[1]{\vert #1\vert}
\newcommand{\floor}[1]{\left\lfloor #1 \right\rfloor}
\newcommand{\st}{~:~}
\newcommand{\threevector}[3]{\begin{bmatrix}
           #1 \\
           #2 \\
           #3
         \end{bmatrix}}

\begin{document}

\begin{abstract}We describe a new continued fraction system in Minkowski space $\mathbb R^{1,1}$, proving convergence, ergodicity with respect to an explicit invariant measure, and Lagrange's theorem. The proof of ergodicity leads us to the question of exactness for products of dynamical systems. Under technical assumptions, namely Renyi's condition, we show that products of exact dynamical systems are again exact, allowing us to study $\alpha$-type perturbations of the system. In addition, we describe new CF systems in $\mathbb R^{1,1}$ and $\mathbb R^{2,1}\cong \Sym_2(\R)$ that, based on experimental evidence, we conjecture to be convergent and ergodic with respect to a finite invariant measure.
\end{abstract}

\maketitle

\section{Introduction}
\label{sec:Introduction}
Regular continued fractions (CFs) describe a point $x\in [0,1)$ as a descending fraction $x=1/(a_1+1/\ldots)$ with $a_i\in \N$. Given $x\in [0,1)$, the \emph{CF algorithm} extracts the digits $a_i$ by iterating the Gauss map $T(x)=1/x - \floor{1/x}$ and setting $a_i=\floor{1/T^{i-1}(x)}$. The reassembled continued fraction then converges to $x$. 

Related algorithms adjust the notion of inversion $\iota$ (e.g., \emph{backwards} CFs), change the digit set $\Zee$ (e.g., \emph{even} CFs), or modify the domain $K$ of the Gauss map  (e.g., \emph{nearest-integer} CFs). Generalizations to higher dimensions have been studied since at least the 19th century work of Hamilton \cite{hamilton1852lii} and Hurwitz \cite{Hurwitz}. The study of generalized CFs has had a resurgence in recent years, see e.g.~Ei-Nakada-Natsui \cite{EI2025110286, Ei-Nakada-Natsui2023}, Kalle-S\'elley-Thuswaldner \cite{kalle2024finiteness}, Lukyanenko-Vandehey \cite{LukyanenkoVandeheySerendipitous, lukyanenko2022ergodicity}, Chousionis-Tyson-Urba\'nski \cite{chousionis2020conformal}, Dani-Nogueira \cite{DN2014}, and Hensley\cite{MR2351741}. 

Here, we extend the study of CFs to spaces without a natural distance function, namely products of CF systems and CFs over Lorentzian spaces. We will be interested in convergence of CF expansions, characterization of points with eventually-repeating expansions, and mixing properties. 

\subsection{Products of exact systems}
\label{subsec:ProductsOfExactSystems}Studies of mixing properties of generalized Gauss maps generally yield either ergodicity or the much stronger exactness property (see \S \ref{subsec:dynamicalpreliminaries} for the definition). Products of ergodic systems need not be ergodic and products of exact systems aren't known to be exact. However, proofs of exactness generally rely on Rokhlin's Exactness Theorem \cite{Rokhlin1961ExactEO}. We prove:

\begin{theorem}
\label{thm:RokhlinProductSystem}
    Suppose the systems $(X_i, \mu_i, T_i)$, with $i\in \{1,2, \ldots, d\}$, satisfy the conditions of Rokhlin's Exactness Theorem \ref{thm:RokhlinExactnessTheorem}. Then their product $(X,\mu,T)=(X_1\times \cdots \times  X_d, \mu_1\times  \cdots \times \mu_d , T_1\times   \cdots \times T_d)$ is exact.
\end{theorem}

As a corollary, one obtains exactness for finite products of systems including the standard Gauss map \cite{Rokhlin1961ExactEO}, its $\alpha$-variants \cite{nakadasteiner}, nearest-integer complex CFs \cite{Nakada1976}, and certain finite-range higher-dimensional systems \cite{LukyanenkoVandeheySerendipitous}.
In particular, we obtain in Theorem \ref{thm:rectangles} exactness for arbitrary shifted CFs $T_{\vec \alpha}: \left([0,1]^d-\vec \alpha \right)\rightarrow  \left([0,1]^d-\vec \alpha \right)$, for $\vec \alpha\in [0,1]^d$. Note that even ergodicity for shifted CFs over the complex numbers remains an open question, see e.g.~\cite{LukyanenkoVandeheySerendipitous, kalle2024finiteness}.

\subsection{Product-type Continued Fractions}
\label{subsec:ProductTypeCF}
Consider two CF systems (see \S \ref{sec:cfs} for a formal definition) representing points in $K_i$ as digits in $\Zee_i$ using a generalized Gauss map $T_i$. The product CF system then represents points in $K_1\times K_2$ using digits in $\Zee_1\times \Zee_2$ using $T_1\times T_2$. If each system is known to be convergent, then so is the product. If each system is known to satisfy the conditions of Rokhlin's Exactness Theorem \ref{thm:RokhlinExactnessTheorem}, then by Theorem \ref{thm:RokhlinProductSystem} the product system is exact. 

For products of \emph{regular} CFs, which we will use in studying Lorentz CFs, we are able to provide more of the standard results by adding a multiplicative structure. View $\R^d$ as a product algebra with coordinate-wise addition and multiplication. The product of coordinate-wise Gauss maps on $[0,1]$ can then be written as a single  mapping on $[0,1]^d$  as $T(x)=(T(x_i))=(\floor{1/x_i})=\floor{1/x}$. We show:

\begin{theorem}\label{thm:MainProductAlgebras} The CF algorithm on the product algebra $\R^n$ with coordinate-wise Gauss map satisfies:
\begin{enumerate}[(a)]
    \item A point $x=(x_i)\in \R^d$ has a finite expansion if and only if each $x_i$ is rational and furthermore the number of CF digits for each $x_i$ is the same.
    \item A point $x=(x_i)\in \R^d$ has an infinite expansion if and only if each coordinate $x$ is \emph{completely irrational}, i.e.~each coordinate $x_i$ of $x$ is irrational. In particular, almost every point has an infinite expansion.
    \item A point $x\in \R^d$ has an eventually-repeating expansion if and only if it is a root of a non-degenerate quadratic equation $a x^2 + bx + c=0$ with coefficients in the product algebra $\Z^d$.
    \item For a completely irrational point $x$, one has $x=1/(a_1 + 1/\ldots)$, in the sense of convergence in the Euclidean topology.
    \item The generalized Gauss map leaves invariant the probability measure $$\frac{1}{(\log 2)^{d}} \frac{ 1}{(1+x_1)\cdots(1+x_d)}dx_1 \cdots dx_d$$
    \item The CF mapping is exact, mixing of all orders, and ergodic.
\end{enumerate}
\end{theorem}

Let for $\vec \alpha=(\alpha_i)\in [0,1]^d$, take $K_\alpha=[0,1]^d-\vec \alpha$ and define $T_{\vec \alpha}: K \rightarrow K$ by $T_{\vec \alpha}(x)=1/x-[1/x]_{\vec \alpha}$, where $[x]_{\vec \alpha} = \floor{x+\vec \alpha}-\vec \alpha$. The resulting CF is a product of one-dimensional $\alpha_i$-CFs.

\begin{theorem}
\label{thm:rectangles}
    Let $\alpha\in [0,1]^d$. Then the mapping $T_\alpha$ is exact with respect to a finite invariant measure. Completely irrational points in $K$ have infinite, convergent expansions.
\end{theorem}

\subsection{Lorentzian Continued Fractions}
\label{subsec:LorentzianCF}
Minkowski space $\R^{p,1}$ consists of vectors $x=(x_1,\ldots, x_{p+1})$ endowed with a quadratic form $Q_{p,1}(x_1, \ldots, x_{p+1}) = x_1^2+\ldots+x_p^2 - x_{p+1}^2$, serving as an analog of a norm squared. The natural inversion is given by $\iota_+(x) = x/Q_{p,1}(x)$. Given elements $a_i\in \R^{p,1}$ one can write $\mathbb K \{a_i\}=\iota_+(a_1 + \iota_+(\cdots))$ to define a (formal) continued fraction. To define a CF algorithm in $\R^{p,1}$, we choose a lattice $\Zee \subset \R^{p,1}$ and a fundamental domain $K$ for $\Zee$, giving a rounding map $\floor{\cdot}:\R^{p,1}\rightarrow \Zee$ characterized by $x - \floor{x}\in K$. We then have a generalized Gauss map $T_+(x) = \iota_+(x) - \floor{\iota_+(x)}$ and digits $a_i = \floor{\iota_+(T^{i-1}x)}$. We provide some general results in this setting in \S \ref{subsec:generalMinkowski}.

Specializing to to the planar case, let $K = \{(x_1,x_2) : |x_1-\frac12| + |x_2| \leq \frac12\}$ be the diamond-shaped region between $0$ and $e_1$, and let $\Zee = \frac{e_1+e_2}{2}\Z \oplus \frac{e_1-e_2}{2}\Z$. We refer to the CF system $(\R^{1,1}, \iota_+, \Zee, K)$ as the Minkowski little-diamond CF.

\begin{theorem}\label{thm:LittleDiamond} The Minkowski little-diamond CF algorithm on $\R^{1,1}$ satisfies:
\begin{enumerate}[(a)]
    \item A point $x=(x_1, x_2)\in \R^{1,1}$ has an infinite expansion if and only if  $x$ is \emph{diagonally-completely irrational}, i.e.~both $(x_1+x_2)$ and $(x_1-x_2)$ are irrational.
    \item Identifying $\R^{1,1}$ with $\R[\jj]=\{x_1+x_2\jj  \st \jj^2=1\}$, $x\in \R^{1,1}$ has an eventually-repeating expansion if and only if it is a root of a non-degenerate quadratic equation $a x^2 + bx + c=0$ with coefficients in $\Zee$.
    \item For a completely irrational point $x$, one has $x=1/(a_1 + 1/\ldots)$, in the sense of convergence in the Euclidean topology, and therefore also relative to $Q_{1,1}$.
    \item The generalized Gauss map leaves invariant the probability measure \(\frac{2}{\left(\log 2\right)^2} \frac{dx_1 dx_2}{(1+x_1)^2-x_2^2}\)
    \item The generalized Gauss map has full cylinders and satisfies Renyi's condition (that is, condition (3) of Theorem \ref{thm:RokhlinExactnessTheorem}), and therefore is exact, mixing of all orders, and ergodic.
\end{enumerate}
\end{theorem}

We will prove Theorem \ref{thm:LittleDiamond} in two steps. First, we observe that the inversion $\iota_c(x_1,x_2) = \frac{(x_1,-x_2)}{x_1^2 - x_2^2}$ results in a CF system with the same cylinders. Thus, it is enough to study the CF system associated with $\iota_c$. Then we show that this system is in fact isomorphic to the product of two regular CF systems, so that Theorem \ref{thm:LittleDiamond} follows from Theorem \ref{thm:MainProductAlgebras}.

Theorem \ref{thm:rectangles} likewise implies that $\alpha$-variants of the $\iota_c$ Minkowski little-diamond CF are convergent and exact for $|\alpha_1|+|\alpha_2|\leq 1$. When $\alpha_1=\alpha_2$ we furthermore get convergence and exactness for the $\iota_+$ Minkowski little-diamond CF.

\subsection{Acknowledgments}  
The second author acknowledges the support of the Natural Sciences and Engineering Research Council of Canada (NSERC), [funding reference numbers RGPIN-2020-05557]

\section{Product-Type Dynamical Systems}
\label{sec:ProductTypeDynamicalSystems}
\subsection{Preliminaries}
\label{subsec:dynamicalpreliminaries}

Throughout this paper, a dynamical system $T:X\rightarrow X$ is a measurable mapping on a measure space $(X, \Sigma, \mu)$, possibly defined away from a set of measure zero. We will assume that $\mu$ is a $T$-invariant probability measure. When we discuss continued fractions, $\mu$ will also be equivalent to the relevant Lebesgue measure $\mu_L$ in the sense that  $C^{-1}\mu \leq \mu_L \leq C \mu$ for some $C\geq 1$.

Given two dynamical systems $(T_i, X_i,\mu_i)$, the product dynamical system is the mapping $(T_1\times T_2)(x_1,x_2)=(T_1(x_1), T_2(x_2))$ on the Cartesian product $X_1\times X_2$ with the product $\sigma$-algebra and product measure $\mu_1\times \mu_2$. We will need three properties of the product measure (see, e.g., \cite{RoydenRealAnalysisEd3}, Chapter 12, and \cite{cohn2015measure}, Theorem 5.1.4):

\begin{lemma}\label{lemma:productmeasures}
Let $(X_1, \Sigma_1,\mu_1)$ and $(X_2, \Sigma_2, \mu_2)$ be two measure spaces. Then for a set $A$ in the $\sigma$-algebra generated by $\Sigma_1$ and $\Sigma_2$, one has:
\begin{enumerate}
\item If $A=E\times F$ for $E\in \Sigma_1$ and $F\in \Sigma_2$, then $(\mu_1\times \mu_2)(E\times F)=\mu_1(E)\mu_2(F)$.
\item $(\mu_1\times\mu_2)(A)=\inf \sum_{i=1}^\infty\mu_1(E_n)\mu_2(F_n)$, where the infimum is taken over all covers of $A$ by products $E_n \times F_n$ with $E_n\in \Sigma_1$ and $F_n\in \Sigma_2$,
\item $(\mu_1\times\mu_2)(A)=\int_{X_1}\mu_2(A_{x_1})\dd\mu_1(x_1)=\int_{X_2}\mu_1(A^{x_2})\dd\mu_2(x_2)$
    where $A_{x_1}=\{x_2\in X_2:(x_1,x_2)\in A\}$ and $A^{x_2}=\{x_1\in X_1:(x_1,x_2)\in A\}$.
\end{enumerate}
\end{lemma}

The dynamical system $T$ is \emph{ergodic} or \emph{irreducible} if, for any measurable set $A\subset X$ one has that $T^{-1}A=A$ implies $\mu(A)=0$ or $\mu(X\setminus A)=0$. Products of ergodic systems need not be ergodic: the negation mapping on $\{-1,1\}$ is ergodic, but its product with itself is negation on the 4-point system $\{(\pm 1, \pm 1)\}$, which is not ergodic since $\{\pm(1,1)\}$ is an invariant set. Ergodicity of $T\times T$ is equivalent to the \emph{weak mixing} condition: for any measurable sets $A,B$, $\lim_{n\rightarrow \infty} \frac{1}{n}\sum_{i=1}^n \mu(T^{-i}A \cap B) = \mu(A)\mu(B)$. Weakly mixing systems are ergodic, and products of weakly mixing systems are again weakly mixing. Stronger yet, a dynamical system $T$ is \emph{exact} if $\mu(A)>0$ implies $\lim_{n\rightarrow \infty} \mu(T^n A) = 1$. Equivalently, tail events have measure zero: denoting by $M$ the $\sigma-$algebra of measurable sets and by $N$ the $\sigma-$algebra of null and co-null sets, one has $\bigcap_{n>0} T^{-n} M = N$. Exactness implies weak mixing (indeed, strong mixing of all orders) and therefore ergodicity. See \cite{cornfeld2012ergodic} for more details.

\subsection{Exactness Theorems}
\label{subsec:ExactnessTheorems}
It appears to be unknown whether a product of exact systems is exact. In practice, exactness is proven using the following theorem that combines the existence of expanding regions (conditions (1) and (2)) with a bound on the expansion rate (condition (3), known as Renyi's condition):
\begin{theorem}[Rokhlin's Exactness Theorem \cite{Rokhlin1961ExactEO}]
\label{thm:RokhlinExactnessTheorem}
    Let $T$ be a dynamical system on a probability space with an invariant measure $\mu$. Suppose there exist:
    \begin{enumerate}
        \item  a collection of measurable sets $\mathcal A$ such that for any measurable $B$ and any $\epsilon>0$, there is a sequence of disjoint sets $A_i\in \mathcal A$ such that  $\mu\left(B \Delta \bigsqcup_i A_i\right)<\epsilon$,
        \item a function $n: \mathcal A \rightarrow \N$ such that $\mu(T^{n(A)}A)=1$,
        \item $\lambda>0$ such that for any $A\in \mathcal A$ and measurable $B\subset A$ one has \(\mu(T^{n(A)}B)\leq \lambda \frac{\mu(B)}{\mu(A)}.\)
    \end{enumerate}
    Then $T$ is exact.
\end{theorem}
\begin{proof}
    Let $B$ be a measurable set satisfying $\mu(B)>0$ and $\lambda>\epsilon>0$. Because $\mu$ is $T$-invariant, $\mu(T^n B)$ forms a non-decreasing sequence, so it suffices to find $n$ such that $\mu(T^n B)>(1-\epsilon)$. 
    Choose  a disjoint collection of sets $A_i \in \mathcal A$ such that $A=\bigsqcup_i A_i$ satisfies $\mu(A \Delta B) < \frac{\epsilon}{4\lambda} \mu(B) <  \frac{1}{2} \mu(B)$. In particular, this gives $\mu(A\setminus B)<\frac{\epsilon}{4\lambda}\mu(B)$ and $\mu(A)>\frac{1}{2}\mu(B)$. 
    Suppose by way of contradiction that for each $i$ one has $\mu(A_i \setminus B)>\frac{\epsilon}{\lambda}  \mu(A_i)$. This would give:
    \begin{align*}\frac{\epsilon}{4\lambda} \mu(B) &> \mu(A\setminus B) = \sum_i \mu(A_i\setminus B) >
    \sum_i \frac{\epsilon}{\lambda} \mu(A_i) =  \frac{\epsilon}{\lambda} \mu(A) \geq  \frac{\epsilon}{2\lambda}\mu(B),
    \end{align*}
    a contradiction. Thus, for some $A_i$ one has $\mu(A_i\setminus B) \le \frac{\epsilon}{\lambda}  \mu(A_i)$. For this $A_i$, we have $\mu(T^{n(A_i)}(A_i\setminus B)) \leq \lambda\frac{\mu(A_i\setminus B)}{\mu(A_i)}<\epsilon$, and so  $\mu(T^{n(A_i)}B)\geq \mu(T^{n(A_i)}(A_i \cap B))\geq (1-\epsilon)$, as desired.
\end{proof}

\begin{remark}
When working with continued fraction systems (see Section \ref{sec:cfs} below), one generally takes $\mathcal A$ to be the collection of full cylinders $C_s$ and $n(C_s)=|s|$, see e.g. \cite{Rokhlin1961ExactEO, Nakada1976, nakadasteiner}.
\end{remark}

We now prove Theorem \ref{thm:RokhlinProductSystem}, which asserts the exactness of a product of systems satisfying the assumptions of Theorem \ref{thm:RokhlinExactnessTheorem}. The proof is a  variation on the proof of Theorem  \ref{thm:RokhlinExactnessTheorem}.

\begin{proof}[Proof of Theorem \ref{thm:RokhlinProductSystem}]
For simplicity, we work with $m=2$. We assume that the systems satisfy the conditions of Rokhlin's Exactness Theorem with respect to the families $\mathcal A_i$, functions $n_i: \mathcal A_i \rightarrow \N$, and constants $\lambda_i>0$.

Let $M\subset X_1 \times X_2$ be a measurable set with $\mu(M)>0$. We would like to show that $\mu(T^nM)\rightarrow 1$. The measure $\mu$ is invariant under $T_1\times T_2$, so that $\mu (T^nM)$ forms an increasing sequence, and we need only show that for any given $\epsilon>0$ there exists $n$ such that $\mu(T^nM)>(1-\epsilon)$.

By Lemma \ref{lemma:productmeasures}, we have $\mu(M)=\inf \sum_j \mu_1(M^j_1)\mu_2(M^j_2)$, where the infimum is taken over all countable covers of $M$ by the rectangular sets $M^j_1\times M^j_2$. By hypothesis (1) of Theorem \ref{thm:RokhlinExactnessTheorem}, we may approximate measurable sets in $X_i$ by elements of $\mathcal A_i$, writing instead $\mu(M)=\inf \sum_j \mu_1(A^j_1)\mu_2(A^j_2)$ where the infimum is taken over all countable covers of $M$ by the rectangular sets $A^j_1\times A^j_2$, with $A^j_1\in \mathcal A_1$ and $A^j_2\in \mathcal A_2$.  Choose such a cover satisfying 
\(0\leq \sum_j \mu_1(A^j_1)\mu_2(A^j_2) - \mu(M) \leq  \epsilon \frac{ \mu(M)}{\lambda_1 \lambda_2}.\)
Suppose, by way of contraction, that for all $j$ we have
\(\frac{\mu(M\cap A^j_1\times A^j_2)}{\mu(A^j_1\times A^j_2)}< 1-\epsilon \frac{1}{\lambda_1 \lambda_2}.\)
Then by subadditivity of measures we have
\begin{align*}
    \mu(M)&\leq \sum_j \mu(M \cap A^j_1\times A^j_2) 
    \leq \sum_j \left(1-\epsilon \frac{1}{\lambda_1 \lambda_2}\right) \mu(A^j_1\times A^j_2)
    \\&= \left(1-\epsilon \frac{1}{\lambda_1 \lambda_2}\right)\sum_j  \mu(A^j_1\times A^j_2)
    \leq \left(1-\epsilon \frac{1}{\lambda_1 \lambda_2}\right)\left(\mu(M)+ \epsilon \frac{ \mu(M)}{\lambda_1 \lambda_2}\right)
    \\&\leq \mu(M)+\epsilon \frac{ \mu(M)}{\lambda_1 \lambda_2} - \epsilon \frac{\mu(M)}{\lambda_1 \lambda_2} - \epsilon \frac{1}{\lambda_1 \lambda_2}\epsilon \frac{ \mu(M)}{\lambda_1 \lambda_2}
\end{align*}
This reduces to $0 \leq  - \frac{ \epsilon^2\mu(M)}{(\lambda_1 \lambda_2)^2}$, a contradiction.
Thus, there exists $j$ such that we have $\frac{\mu(M\cap A^j_1\times A^j_2)}{\mu(A^j_1\times A^j_2)} \geq 1 - \epsilon \frac{1}{\lambda_1\lambda_2}$.  Letting $A = A^j_1 \times A^j_2$, we conclude that $\mu(A\setminus M) \leq \epsilon \frac{\mu(A)}{\lambda_1\lambda_1}$.

Write $n_i = n(A_i)$ and assume, without loss of generality, that $n_1<n_2$. Observe that 
$T^{n_2} = (T_1^{n_2-n_1}\times \id) \circ (\id\times T_2^{n_2})\circ (T_1^{n_1}\times\id)$. We apply each of these in sequence, noting first that by definition of $n_1$ we have that $\mu_1(T^{n_1}(A_1))=1$. 
By Fubini's Theorem and Rokhlin's condition (3), the product mapping $(T_1^{n_1}\times\id)$ increases relative volume by a factor of at most $\lambda_1$ over $A_1$, so that 
\(\frac{\mu( (T_1^{n_1}\times\id)(A\setminus M))}{\mu((T_1^{n_1}\times\id)(A))}=\frac{\mu( (T_1^{n_1}\times\id) (A\setminus M))}{1\cdot \mu_2(A_2)} \leq \lambda_1 \frac{ \mu(A\setminus M)}{\mu_1(A_1)\mu_2(A_2)}\)
and likewise 
\(\frac{\mu((\id\times T_2^{n_2}) (T_1^{n_1}\times\id) (A\setminus M))}{1\cdot 1} \leq \lambda_1\lambda_2 \frac{\mu(A\setminus M)}{\mu(A)} \leq  \epsilon.\)
Thus, we already have $\mu((T_1^{n_1}\times T_2^{n_2})(A\cap M)) \geq (1-\epsilon)$. Because $\mu_1\times \mu_2$ is invariant under $(T_1\times \id)$ we obtain that $\mu((T_1^{n_2-n_1}\times \id) \circ (\id\times T_2^{n_2})\circ (T_1^{n_1}\times\id)(A\cap M))=\mu(T^{n_2}(A\cap M))\geq (1-\epsilon)$ and thus $\mu(T^{n_2}(M))>(1-\epsilon)$.
\end{proof}

\section{Preliminaries on continued fractions}
\label{sec:cfs}
\subsection{CF framework}
\label{subsec:CFframework}
We will work with generalized continued fractions in a framework analogous to that of Iwasawa CFs \cite{lukyanenko2022ergodicity}. Namely, we will fix the \emph{CF algorithm data}:
\begin{enumerate}
    \item a space $\X = \R^d$ equipped with the standard topology, the standard smooth structure, and a measure equivalent to Lebesgue measure, denoted by $\mu$,
    \item a \emph{null space} $\mathcal N \subset \X$ such that $\mu(\mathcal N)=0$ and $0\in \mathcal N$,
    \item a smooth order-two \emph{inversion} mapping $\iota: \X\setminus \mathcal N \rightarrow \X \setminus \mathcal N$,
    \item a lattice $\Zee\subset \R^d$ serving as the \emph{digits}
    of the CF, which we will interpret as mappings $a: \R^d\rightarrow \R^d$,
    \item a fundamental domain $K\subset \X$ for $\Zee$ containing $0$, i.e., $X=\bigsqcup_{a\in \Zee}(a+K)$ and furthermore $K$ is contained in the closure of its interior and $\mu(\partial K)=0$,
    \item a rounding function $[\cdot]_K: \X \rightarrow \Zee$ associated to $K$, defined by the relation $[x]_K^{-1}(x)\in K$.
\end{enumerate}

For a fixed set of CF algorithm data, we then define:
\begin{enumerate}
    \item A finite continued fraction is an expression of the form $a_1 \iota a_2 \iota \cdots \iota a_n(0)$ (where we suppress composition notation),
    \item An infinite CF is a (possibly formal) limit $\lim_{n\rightarrow \infty} a_1 \iota a_2 \iota \cdots \iota a_n(0)$,
    \item The generalized Gauss mapping $T: K \rightarrow K$ is given by $T(x)=[\iota x]_K^{-1}(\iota x)$,
    \item The CF digits of $x\in K$ are the sequence $a_i = a_i(x) = [\iota T^{i-1} x]_K$. This sequence will be finite if $T^n x \in \mathcal{N}$ for some $n$, but may also be infinite.
    \item For a fixed sequence of digits $s=(b_1, \ldots, b_n)$ having length $|s|=n$, the \emph{cylinder set} $C_s$ is the set of points \(C_s = \{x \in K : a_i(x)=b_i \text{ for each }1\leq i\leq n\}.\)
    \item A cylinder $C_s$ is \emph{full} if $\mu(T^{|s|}(C_s))=1$. 
\end{enumerate}

A continued fraction system is \emph{convergent} if for any $x\in K$ the CF digits $a_i=a_i(x)$ satisfy $\lim_{n\rightarrow \infty} a_1 \iota a_2 \iota \cdots \iota a_n(0) = x$.  It \emph{has full cylinders} if all cylinders with positive measure are full. It is \emph{finite range} if the collection of sets $\{T^{|s|}(C_s)\}$, considered up to measure zero, is finite. It is \emph{ergodic} or \emph{exact} if $T$ is ergodic (resp., exact) with respect to \emph{some} invariant probability measure on $K$ that is equivalent to Lebesgue measure.

\subsection{Regular CFs}
\label{subsec:RegularCFs}
    Regular CFs are given by the data $\X=\R$, $\mathcal N=\{0\}$, $\Zee=\mathbb Z$, $\iota(x)=1/x$, $K=[0,1)$. Rational numbers are represented by finite CFs, while irrational numbers have convergent representations. Lagrange's theorem (see e.g. \cite{MR2351741}) states that eventually-repeating CF representations correspond to roots of non-degenerate quadratic polynomials over $\Z$. The regular CF system has full cylinders and satisfies Renyi's condition. Ergodicity goes back to work of Gauss, and was first formalized by Renyi \cite{MR97374} and extended to exactness by Rokhlin \cite{Rokhlin1961ExactEO}.

\subsection{Tanaka-Ito CFs}
\label{subsec:TanakaItoCF}
    For $\alpha\in [0,1)$, the Tanaka-Ito \cite{Tanaka-Ito} $\alpha$-CFs are given by the data $\X=\R$, $\mathcal N=\{0\}, \Zee=\mathbb Z, \iota(x)=1/x$, $K=[\alpha-1,\alpha)$. For $\alpha\notin\{0,1\}$, the system has cylinders that are not full; when $\alpha$ is a root of a linear or quadratic equation, the collection $\{T^{|s|}C_s\}$ of normalized cylinders is finite. Convergence for $\alpha$-CFs was shown in \cite{Tanaka-Ito}, and it was only recently that the conjectured ergodicity and exactness were confirmed by Nakada-Steiner \cite{nakadasteiner}, using Rokhlin's Exactness Theorem.

\section{Product-type continued fractions}
\label{sec:productCFs}
 
In this section, we study product-type continued fractions on the product algebra $\R^d$. That is, we endow $\R^d$ with the coordinate-wise operations
\begin{align*}(x_1, \ldots, x_d) + (x'_1, \ldots, x'_d) &= (x_1+x'_1, \ldots, x_n+x'_d)\\
(x_1, \ldots, x_d) \times (x'_1, \ldots, x'_d) &= (x_1x'_1, \ldots, x_d x'_d)\end{align*}
Note that the multiplicative identity is $(1,\ldots, 1)$ and that an element $(x_1, \ldots, x_d)$ is invertible if and only if each $x_i$ is non-zero. 

Taking the null set $\mathcal N$ to be the non-invertible points (including the origin $0=(0,\ldots, 0)$), we have a natural inversion $\iota:\R^d\setminus \mathcal N \rightarrow \R^d \setminus \mathcal N$ given by $\iota(x)=x^{-1}$. In coordinates one has $\iota(x_1, \ldots, x_d) = (1/x_1, \ldots, 1/x_d)$.

\subsection{Products of Regular CFs}
\label{subsec:ProductRegularCFs}
Let $\Zee=\Z^d$, $K=[0,1)^d$ and $\iota$ and $\mathcal N$ as above. The associated system is a product of regular CFs in $\R$. We will write $T(x) = (T_\R(x_1), \ldots, T_\R(x_d))$ where $T_\R$ is the Gauss map $T_\R(x) = 1/x - \floor{1/x}$.

\begin{proof}[Proof of Theorem \ref{thm:MainProductAlgebras}]
We prove the assertions of the theorem in an altered order.
\begin{enumerate}
\item[(a)] Given a finite sequence $a_1, \ldots, a_n$ of digits in $\Z^d$, the associated continued fraction $x=\iota a_1 \cdots \iota a_n(0)$ produces an element of $\Q^d$ such that each coordinate has a CF expansion of length $n$. Applying the map $T$ to the resulting point removes the digits $a_i$ one at a time, so that $T^nx=0$. Conversely, given a starting point $x$, the map $T$ can be iterated as long as $T^n_\R(x_i)\neq 0$ for each $i$. If all $x_i$ are rational and have the same number of CF digits, say $n$, then $T^n(x)=0$ and we can reconstruct $x$ from the digits.

\item[(b)] A point $x=(x_i)$ has an infinite expansion if and only if $T^n(x)\notin \mathcal N$ for all $n$. Equivalently, $T^n _\R(x_i)\neq 0$ for all $i$ and $n$, i.e.~each $x_i$ is irrational.

\item[(d)] For completely irrational points $x\in \R^d$, we obtain a convergent CF in each coordinate, so the sequence in $\R^d$ is likewise convergent.

\item[(c)] Observe that in the product algebra $\R^d$ one has that $e_ie_j = \delta_{i,j}e_i$, and consider a quadratic equation with coefficients in $\Q^d$: 
\begin{align*}
   ax^2 + bx+c =
   \threevector{a_1}{\vdots}{a_d}
   \threevector{x_1}{\vdots}{x_d}
   \threevector{x_1}{\vdots}{x_d} + 
   \threevector{b_1}{\vdots}{b_d}
   \threevector{x_1}{\vdots}{x_d}
   + 
   \threevector{c_1}{\vdots}{c_d}
   =
   \threevector{a_1x_1^2 + b_1x_1+c_1}{\vdots}{a_dx_d^2 + b_dx_d+c_d}.
\end{align*}
We see that a quadratic equation $ax^2 + bx+c=0$ in $\R^d$ is simply a collection of equations $a_ix_i^2 + b_ix_i+c_i=0$. Thus, the classical Lagrange's theorem gives a joint Lagrange theorem in $\R^d$.

\item[(e)] The given measure is the product of the invariant measures for $T_\R$ along each coordinate.

\item[(f)] Exactness follows from Theorem \ref{thm:RokhlinProductSystem}. \qedhere
\end{enumerate}
\end{proof}

\subsection{Rectangular CFs}
\label{subsec:RectangularCF}
Let $K_\alpha = [0,1)^d-\alpha$, for $\alpha\in [0,1]^d$, be a rectangular fundamental domain for $\Z^d$ that is parallel to the axes of $\R^d$. The resulting CF splits as a product of $\alpha$-CFs in each coordinate. We thus have:

\begin{theorem}
    Let $K\subset [-1,1]^d$ be any rectangular fundamental domain for $\Z^d$ that is parallel to the axes of $\R^d$. Then the associated CF on the product algebra $\R^d$ with lattice $\Z^d$ and inversion $\iota(x)=x^{-1}$ is convergent and exact.
\begin{proof}
    Convergence follows from the fact that $\alpha$-CFs are convergent \cite{Tanaka-Ito}. Exactness of $\alpha$-CFs was proven in \cite{nakadasteiner} as a consequence of two results: the full cylinders of $T_\alpha$ generate the Borel $\sigma$-algebra of $[\alpha-1, \alpha)$ (\cite[Proposition 1]{nakadasteiner}), and Renyi's condition holds (\cite[Proposition 2]{nakadasteiner}). Together, these imply the conditions of Rokhlin's Exactness Theorem \ref{thm:RokhlinExactnessTheorem}, so the product system is exact by Theorem \ref{thm:RokhlinProductSystem}.
\end{proof}
\end{theorem}

\section{Minkowski Continued Fractions}
\label{sec:MinkowskiCF}
\subsection{General Results}\label{subsec:generalMinkowski}
Let $\X=\R^{p,q}$, i.e.~the set $\R^p\times \R^q$ with the indefinite inner product $\langle (x_1,x_2), (y_1,y_2)\rangle_{p,q}=x_1\cdot y_1 -  x_2 \cdot y_2$ and quadratic form $Q(x)=\langle x,x\rangle_{p,q}$. Take $\mathcal N = \{x: Q(x)=0\}$. Fixing $\mathcal O\in O(p,q)$, write $\iota(x)=\frac{\mathcal O(x)}{Q(x)}$ for $x\not\in \mathcal N$. The mapping $\iota$ is an inversion whenever $\mathcal O^2=\id$. We will primarily be interested in CFs in $\R^{1,1}$, which we study from a different perspective, but we take this opportunity to record some standard results in this general setting. 

\begin{lemma} The mapping $\iota$ satisfies the identities $Q(\iota x)=Q(x)^{-1}$ and $Q(x-y)=Q(x)Q(y)Q(\iota(x)-\iota(y))$.
\begin{proof}
Since $\mathcal O \in O(p,q)$, we have that $Q(\mathcal O(x)) = Q(x)$ for any $x\in \R^{p,q}$, and likewise inner products are preserved by $\mathcal O$. We may therefore simply consider the case $\mathcal O=\id$.

The first identity holds because $Q(cx)=c^2Q(x)$ for $c\in \R$, so that $Q(\iota(x))=Q(x/Q(x))=Q(x)^{-2}Q(x)=Q(x)^{-1}$. For the second identity, one computes:
\begin{align*}
Q&(\iota(x)-\iota(y))=\left \langle \frac{x}{Q(x)}-\frac{y}{Q(y)}, \frac{x}{Q(x)}-\frac{y}{Q(y)}\right \rangle_{p,q}\\
&=Q(x)^{-2}Q(y)^{-2}\langle Q(y)x-Q(x) y, Q(y)x-Q(x) y\rangle_{p,q}\\
&=Q(x)^{-1}Q(y)^{-1}(Q(y)+Q(x)-\langle x,y\rangle_{p,q} -\langle y,x 
\rangle_{p,q})\\
&=Q(x)^{-1}Q(y)^{-1}(\langle x-y,x-y\rangle_{p,q})=Q(x)^{-1}Q(y)^{-1}Q(x-y).\qedhere
\end{align*}
\end{proof}
\end{lemma}

The mapping $\iota$ is thus ``expanding'' relative to $|Q|$ on $B_Q=\{x: |Q(x)|<1\}$:
\begin{corollary}
For $x,y\in B_Q$ one has $|Q(\iota x-\iota y)|> |Q(x-y)|$.
\end{corollary}

The expanding region for the Euclidean metric is more restricted. Denote the Euclidean metric as $d_E$ and let $B_E=\{(x_1,x_2)\in\R^{p,q}:|x_1|+|x_2|<1\}\setminus \mathcal N$. Then we have:

\begin{lemma}
\label{lemma:EuclideanExpandingRegion} The singular values of $\iota$ are bounded below by 1 on $B_E$. Consequently, $\iota$ is an expanding mapping on subsets of $B_E$ with convex image.
\begin{proof}
Observe first that it suffices to study the singular values: if the singular values are bounded below by 1, then $\norm{\dd \iota(\vec v)}\geq \norm{\vec v}$. Thus, if two points in $\iota(B_E)$ are connected by a straight line $\gamma$, then the length of $\iota^{-1}(\gamma)$ is bounded above by the length of $\gamma$.

Observe next that it suffices to compute singular values for $\R^{2,2}$: given a point $x\in \R^{p,q}$ and a vector $v\in \R^{p,q}$ based at $p$, we may apply an element of $O(p)\times O(q)\subset O(p,q)\cap O(p+q)$ to rotate $x$ and $v$ into $\R^{2,2}\subset \R^{p,q}$. This transformation preserves membership in $B_E$, preserves the Euclidean inner product, and commutes with $\iota$, so the singular values for $\iota$ in $\R^{2,2}$ and $\R^{p,q}$ will agree.

Indeed, by symmetry it is sufficient to compute $(\dd\iota)^T \dd\iota$ at a point $(l,0,r,0)\in \R^{2,2}$ with $l,r> 0$, and we may assume $\mathcal O=\id$. One obtains:
    \(\dd\iota\vert_{(l,0,r,0)}=\frac{1}{(l^2-r^2)^2}\begin{bmatrix}
        -l^2-r^2 & 0 & 2lr & 0\\
        0 & l^2-r^2 & 0 & 0\\
        -2lr & 0 & l^2+r^2 & 0\\
        0 & 0 & 0 & l^2-r^2
    \end{bmatrix}\)
    
    Calculating the eigenvalues of $(\dd\iota)^T \dd\iota\vert_{(l,0,r,0)}$ gives singular values of $\dd\iota\vert_{(l,0,r,0)}$:
    \((l+r)^{-2},\hspace{.25in}(l-r)^{-2},\hspace{.25in}|l^2-r^2|^{-1},\)
    corresponding to the vectors $e_1+e_3$, $-e_1+e_3$, and both $e_2$ and $e_4$, respectively.
    
    Given our normalization to $l,r>0$, we have that $l+r>|l-r|>|(l+r)(l-r)|^{1/2}$, so that $(l+r)^{-2}$ is the smallest singular value. The condition $(l+r)^{-2}>1$ corresponds exactly to membership in $B_E$, as desired.
\end{proof}
\end{lemma}

\subsection{Two-dimensional systems}
\label{subsec:2DMinkowskiSystems}
We now focus on the case of $\X=\R^{1,1}$ with coordinates $x=(x_1,x_2)$. We give $\R^{1,1}$ additional algebraic structure by identifying it with the commutative algebra $\R[\jj]=\{x_1+x_2\jj :\jj^2=1\}$. Borrowing notation from complex numbers, we write $\overline{x_1+x_2\jj }=x_1-x_2\jj $,  $Q(x)=\Re(x\overline x)$, and $\iota_+(x)=x/Q(x)=x/(x\overline x) = 1/\overline{x}$. Observe that $\R[\jj]$ is, in fact, isomorphic to the product algebra $\R^2$ (see Section \ref{sec:productCFs}):
\begin{lemma}
\label{lemma:changeOfCoordinates}
The mapping $\Phi: \R^2 \rightarrow \R[\jj]$ given by $\Phi(a,b)=a \frac{1+\jj }{2} + b \frac{1-\jj }{2}$ is an algebra isomorphism with inverse $\Phi^{-1}(x_1,x_2)=(x_1+x_2,x_1-x_2)$. Furthermore, one has $\overline{\Phi(a,b)}=\Phi(b,a)$, $\Re(\Phi(a,b))=\frac{a+b}{2}$, and $Q(\Phi(a,b))=ab$. 
\begin{proof}The mapping $\Phi$ is linear and non-degenerate, so it is a vector space isomorphism between $\R^2$ and $\R[\jj]$. The multiplicative structure of $\R^2$ is characterized by the property that $(1,0)^2=(1,0)$, $(0,1)^2=(0,1)$ and $(1,0)*(0,1)=(0,0)$. One checks that $\Phi(1,0)^2 = \left(\frac{1+\jj }{2}\right)^2 = (\frac{2+2\jj}{4})=\Phi(1,0)$. Likewise, $\Phi(0,1)^2 = \Phi(0,1)$ and $\Phi(1,0)*\Phi(0,1)=\Phi(0,0)$. Thus, $\Phi$ is a bijective algebra homomorphism, hence an isomorphism. The identities $\overline{\Phi(a,b)}=\Phi(b,a)$ and $\Re(\Phi(a,b))=\frac{a+b}{2}$ are immediate from the definition of $\Phi$, and give $Q(\Phi(a,b))=\Re (\Phi(a,b)\overline{\Phi(a,b)})=\Re(\Phi(ab, ab))=ab$.
\end{proof}

\end{lemma}

Before proving Theorem \ref{thm:LittleDiamond}, we prove a variant of it with the more natural inversion $\iota_c(x) = \overline x/Q(x)$.
\begin{theorem}\label{thm:zbar} Consider the CF algorithm on $\R[\jj]$ given by the inversion $\iota_c(x) = \overline x/Q(x)$, lattice $\Zee = \Z \frac{1+\jj }{2} \oplus \Z \frac{1-\jj }{2}\subset \R[\jj]$, and fundamental domain given by the little diamond $K=\{x_1+x_2\jj \in \R[\jj] : |x_1-\frac12|+|x_2|\leq \frac12\}$. Then:
\begin{enumerate}[(a)]
    \item A point $x=x_1+x_2\jj \in \R[\jj]$ has an infinite expansion if and only if  $x$ is \emph{diagonally-completely irrational}, i.e.~both $(x_1+x_2)$ and $(x_1-x_2)$ are irrational.
    \item A point $x\in \R[\jj]$ has an eventually-repeating expansion if and only if it is a root of a non-degenerate quadratic equation $a x^2 + bx + c=0$ with coefficients in $\Zee$.
    \item For a completely irrational point $x$, one has $x=1/(a_1 + 1/\ldots)$, in the sense of convergence in the Euclidean topology, and therefore also relative to $Q$.
    \item The generalized Gauss map leaves invariant the probability measure \(\frac{2}{\left(\log 2\right)^2} \frac{dx_1 dx_2}{(1+x_1)^2-x_2^2}\)
    \item The generalized Gauss map has full cylinders and satisfies Renyi's condition, and therefore is exact, mixing of all orders, and ergodic.
\end{enumerate}
\begin{proof}
Let us refer to the above CF algorithm as the $\iota_c$-CFs, with notation $\Zee_{c}$, $\iota_c$, $K_{c}$. Consider also the CF system on $\R^2$ given by the product of regular CFs on each coordinate; we will refer to these as the $\R^2$-CFs, with notation $\Zee_{\R^2}$, $\iota_{\R^2}$, $K_{\R^2}$.

We first claim that the algebra isomorphism $\Phi: \R^2 \rightarrow \R[\jj]$ identifies $\R^2$-CFs with $\iota_c$-CFs. Of course, it identifies the spaces, the topologies, and the measure classes. It is immediate from the definition of $\Zee_{c}$ that $\Phi(\Zee_{\R^2}) = \Zee_{c}$. The inversions in both algebras can be expressed as $1/x$, and thus are identified by $\Phi$: 
\(\iota_{\R^2}(a,b) = (1/a,1/b)=1/(a,b), \hspace{.5in} \iota_c(x) = \overline{x}/Q(x) = \overline x/ (x\overline x) = 1/x.\)
Lastly, $K_c$ is a diamond with vertices $\{0, \frac{1+\jj }{2}, \frac{1-\jj }{2}, 1\} = \Phi(\{(0,0),(1,0), (0,1), (1,1)\})$. Thus, $K_c = \Phi(K_{\R^2})$. This provides an identification of all the ingredients of the CF algorithm enumerated in \S \ref{subsec:CFframework}.

\begin{figure}
\includegraphics[width=.3\textwidth]{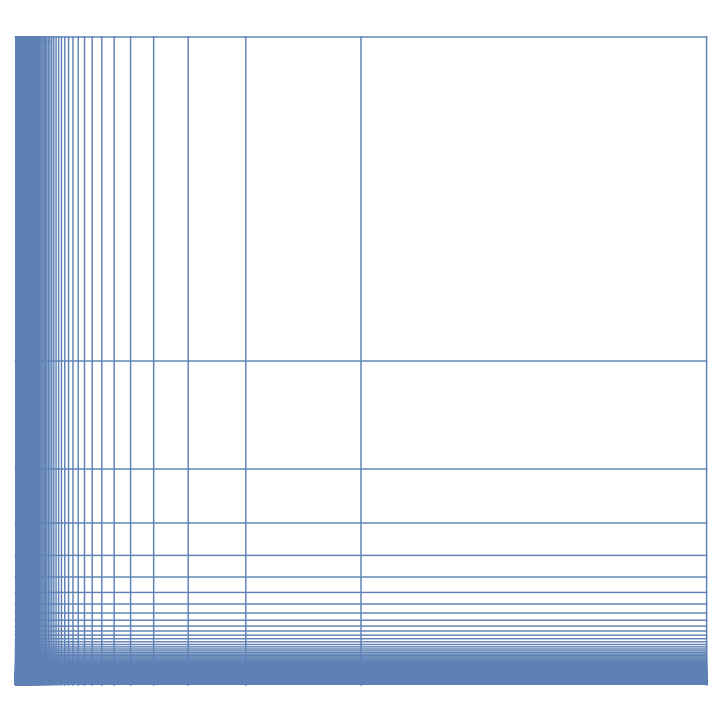}
\hspace{.1\textwidth}
\includegraphics[width=.3\textwidth]{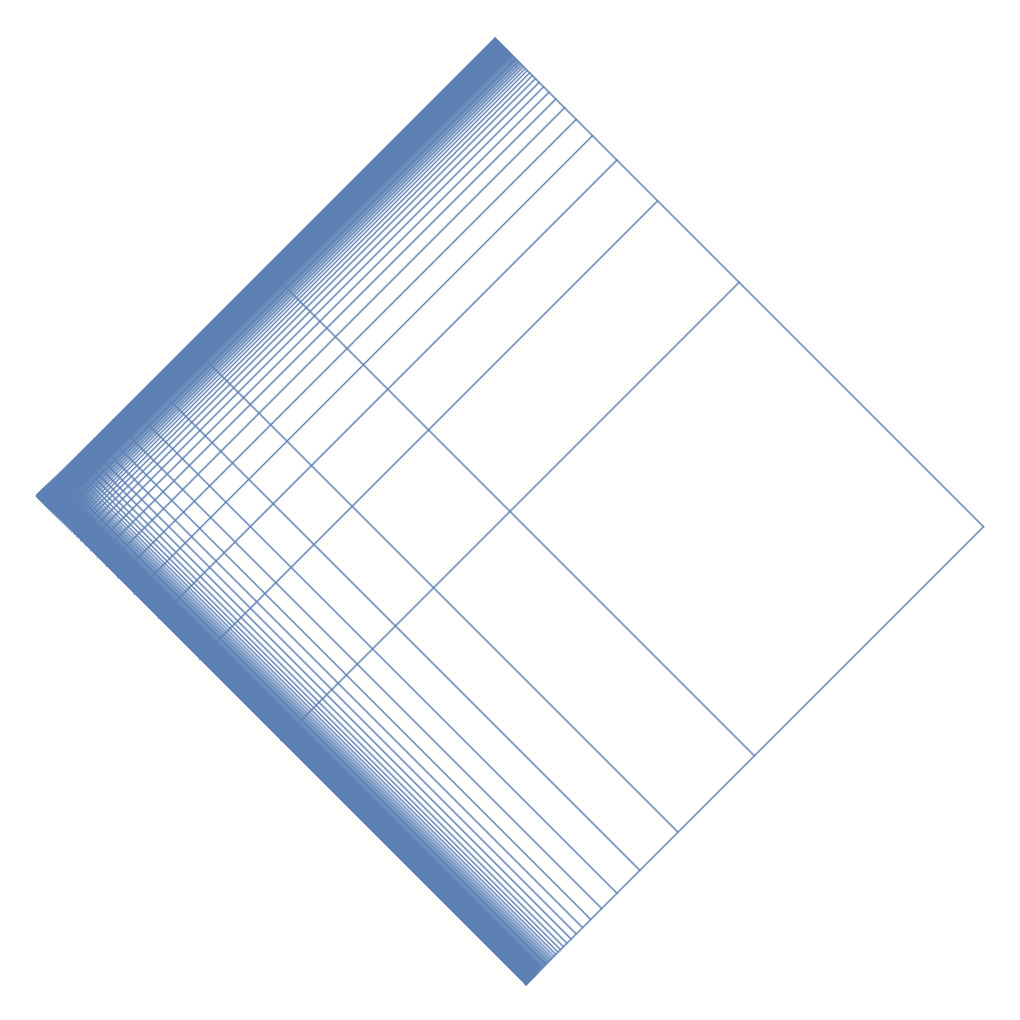}
\caption{Left: depth-1 cylinders in $[0,1)^2$ for the product CFs in $\R^2$. Right:  depth-1 cylinders for the little diamond CFs in $\R[\jj]$.}
\end{figure}

We now claim that the conclusions of the theorem follow from Theorem \ref{thm:MainProductAlgebras}.
\begin{enumerate}[(a)]
\item By Theorem \ref{thm:MainProductAlgebras}(b), $(a,b)\in \R^2$ has an infinite expansion if and only if both $a$ and $b$ are irrational. Thus, $x=x_1+x_2\jj $ has an infinite $\iota_c$-CF expansion if and only if both coordinates of $\Phi^{-1}(x_1+x_2\jj ) = (x_1+x_2, x_1-x_2)$ are irrational.
\item The statement is equivalent to Theorem \ref{thm:MainProductAlgebras}(c), under the identifications provided by $\Phi$.
\item Since $\Phi$ is a homeomorphism, convergence of the convergents in the Euclidean topology is preserved. This, again, gives convergence with respect to $Q$.
\item The invariant measure for $T_{\R^2}$ is given, in differential form notation, by
\begin{align*}
&\frac{1}{(\log 2)^2}\frac{|\dd a\wedge \dd b|}{(1+a)(1+b)}  = \frac{1}{(\log 2)^2}\frac{|\dd(x_1+x_2)\wedge \dd (x_1-x_2)|}{(1+x_1+x_2)(1+x_1-x_2)}\\&= 
 \frac{1}{(\log 2)^2}\frac{|-2 \dd x_1\wedge \dd x_2|}{(1+x_1+x_2)(1+x_1-x_2)}= \frac{2}{(\log 2)^2}\frac{\dd x_1 \dd x_2}{(1+x_1)^2-x_2^2}.
\end{align*}
 \item The fact that the generalized Gauss map has full cylinders follows from the fact that it's conjugate, via $\Phi$, to a product of classical Gauss maps, which have full cylinders. This is enough to conclude that it is exact by Theorem \ref{thm:RokhlinProductSystem}, since the classical Gauss map satisfies the conditions of Rokhlin's Theorem \ref{thm:RokhlinExactnessTheorem}. Additionally, we obtain that the product system still satisfies Renyi's condition, which respects products when written in its differential form: for some $\lambda >0$ one has $\frac{\sup_{x\in A} J_x(T^{n(A)})}{\inf_{x\in A} J_x(T^{n(A)})}< \lambda$, where $J_xT^{n(A)} = \left|\det \dd T^{n(A)}\vert_x\right|$.\qedhere
\end{enumerate}
\end{proof}
\end{theorem}

We can now prove Theorem \ref{thm:LittleDiamond} as a corollary of Theorem \ref{thm:zbar}.

\begin{proof}[Proof of Theorem \ref{thm:LittleDiamond}]
Let us refer to the Little Diamond CF of Theorem \ref{thm:LittleDiamond} as the $\iota_+$-CF. The $\iota_+$-CF uses the inversion $\iota_+(x)=x/Q(x)$ but otherwise has identical data to the $\iota_c$-CF of Theorem \ref{thm:zbar}, which uses the inversion $\iota_c(x)=\overline x/Q(x)$.  We therefore recover all the desired conditions from Theorem \ref{thm:zbar}:
\begin{enumerate}
\item If one of the expansions is given by the sequence $(a_1, a_2, a_3, \ldots)\subset \Zee$, then the other expansion for the same point is given by $(\overline{a_1}, a_2, \overline{a_3}, \ldots)$. Thus, one sequence is infinite if and only if the other is, and by Theorem \ref{thm:zbar} is equivalent to $x$ being diagonally-completely irrational.
\item The digit sequence $(a_1, a_2, a_3, \ldots)$ is eventually-periodic if and only if the same holds for $(\overline{a_1}, a_2, \overline{a_3}, \ldots)$ (perhaps with a doubling of the period). By Theorem \ref{thm:zbar}, this is equivalent to $x_1+x_2\jj $ being a solution to a non-degenerate quadratic equation with coefficients in $\Zee$.
\item The convergents are identical for the two CF systems, so convergence for $\iota_+$-CFs follows from that of $\iota_c$-CFs.
\item Denote the generalized Gauss maps by $T_+$ and $T_c$, and conjugation by $c(x)=\overline x$. Then $c$ commutes with both $T_+$ and $T_c$ and $T_+ =c \circ T_c$. The invariant measure $\mu$ for $T_c$ given in Theorem \ref{thm:zbar} is $c$-invariant, and is therefore also an invariant measure for $T_+$:
\(\mu(T^{-1}_+(A)) = \mu(c\circ T^{-1}_c A) =  \mu(T^{-1}_c A)=  \mu(A).\)
\item Because of symmetry, the cylinders for $\iota_+$-CFs agree with cylinders for $\iota_c$-CFs, and we obtain Renyi's condition because $J_x T_{+}^{n(A)}=J_x T_{c}^{n(A)}$, since $c$ preserves Lebesgue measure. We thus obtain exactness of the $\iota_+$-CFs from Theorem \ref{thm:RokhlinExactnessTheorem}. Alternately, exactness follows from Theorem \ref{thm:zbar} by using the fact that $T^2_+=T^2_c$ so that for any set $A$ with $\mu(A)>0$ we have $\mu(T_+^n(A))\rightarrow 1$ if and only if $\mu(T_c^n(A))\rightarrow 1$.\qedhere
\end{enumerate}
\end{proof}

Analogously to $\alpha$-CFs in $\R$ and $\R^d$, one can define $\alpha$-CFs in $\R[\jj]$. These are well-behaved for the inversion $\iota(z)=1/z$:

\begin{theorem}
Fix $\alpha=\alpha_1+\alpha_2\jj \in \R[\jj]$ with $|\alpha_1|+|\alpha_2|<1$. Consider the CF algorithm on $\R^{1,1}$ given by the inversion $\iota(x) = \overline x/Q(x)$, lattice $\Zee = \Z \frac{1+\jj}{2} \oplus \Z \frac{1-\jj}{2}\subset \R[\jj]$, and fundamental domain given by the $\alpha$-shifted little diamond $K=-\alpha+\{x+\jj y\in \R[\jj] : |x_1-\frac12|+|x_2|\leq \frac12\}$. The resulting $CF$ algorithm is convergent and exact, with a finite invariant measure equivalent to Lebesgue measure.
\begin{proof}
Lemma \ref{lemma:changeOfCoordinates} transforms the above CF into a rectangular CF in $\R^2$, and the result follows from Theorem \ref{thm:rectangles}.
\end{proof}
\end{theorem}

\section{Experimental evidence and open questions}\label{sec:OpenQuestions}
\subsection{Two-dimensional systems}
\label{subsec:2DSystems}
While we are able to obtain results for systems that are closely related to products of known CFs, there remain natural CF algorithms in $\R^{1,1}$ that are not as easily analyzed. We present some examples here.
\subsubsection{Square CFs} 
\label{subsubsec:SquareCF}
Following the example of Hurwitz complex CFs \cite{MR2351741}, one can define a CFs in $\R^{1,1}$ with inversion $\iota(x)=x/Q(x)$, digits in $\Zee=\Z^2$ and fundamental domain $K=[-1/2,1/2]\times [-1/2,1/2]$, which is contained in the Euclidean-expanding region. Let us refer to these CFs as the square CFs. Experimental evidence suggests that this CF algorithm is convergent with respect to both $Q$ and the Euclidean metric, and ergodic with respect to a probability measure equivalent to Lebesgue measure, see Figure \ref{fig:squareCF}. However, it is unclear how to study this dynamical system. Standard convergence arguments fail (crucially, $Q(x)=0$ does not imply $x=0$). Additionally, the system does not satisfy either the full-cylinder or the finite range condition: because $\iota(\partial K)$ is a pair of hyperbolas with slope arbitrarily close to 1, the collection $\{T(C_{\{(n,n+1)\}}):n\in \N_+\}$ is infinite.

\begin{figure}[ht]
\includegraphics[width=.45\textwidth]{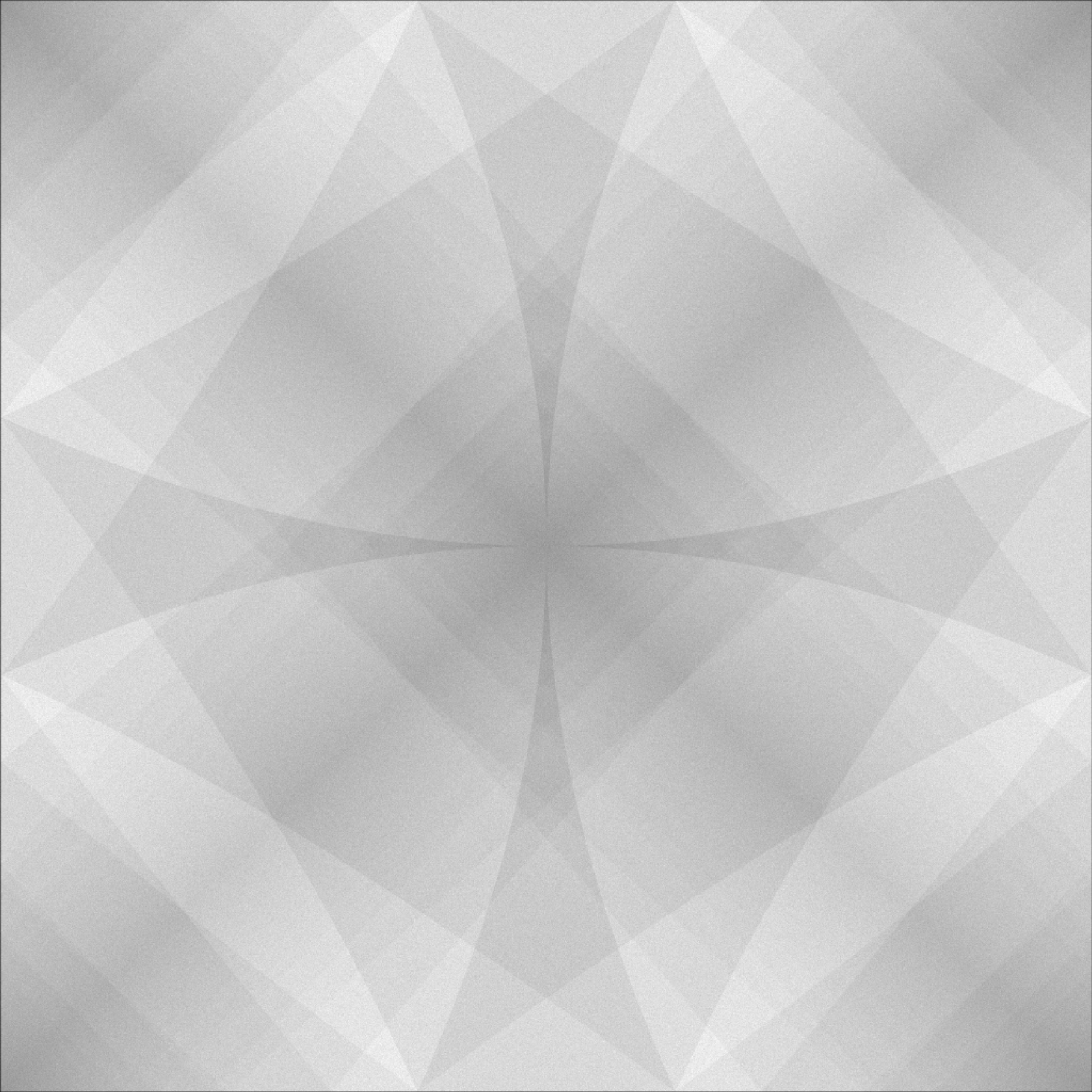}
\includegraphics[width=.45\textwidth]{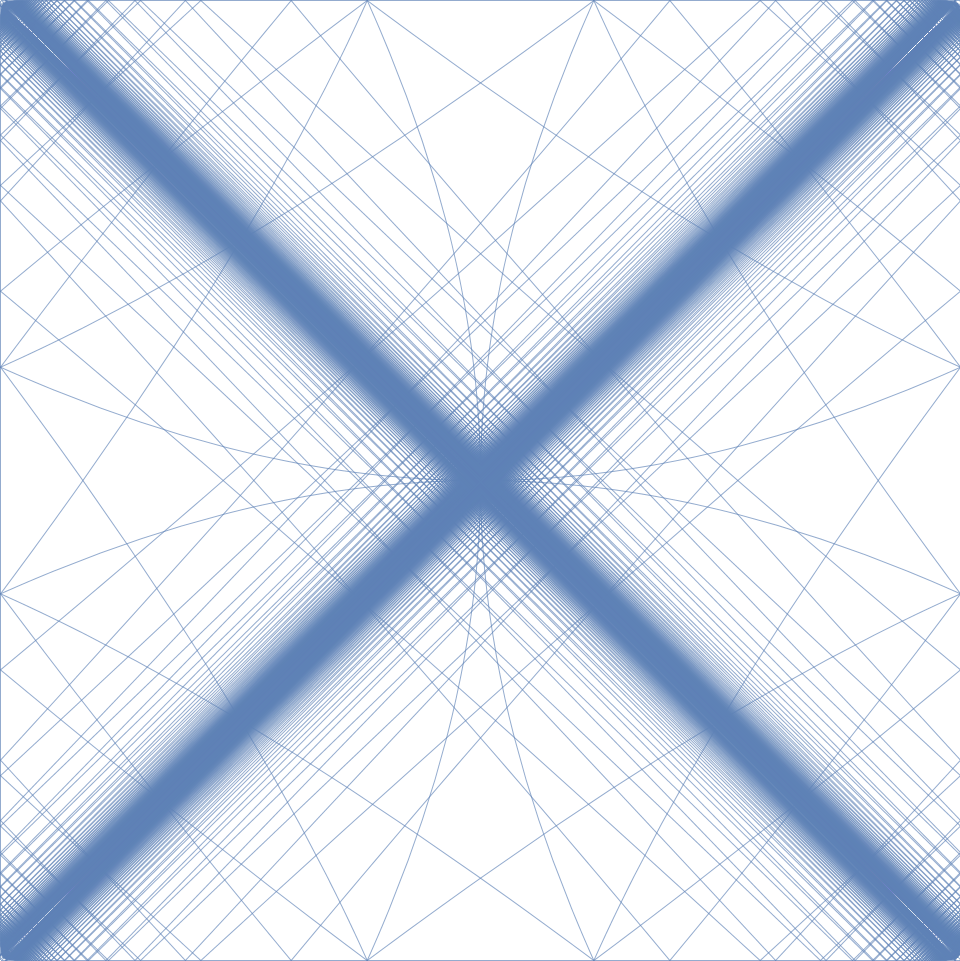}
\caption{Left: Experimental evidence suggests that square CFs in $\R^{1,1}$ have a finite invariant measure. Right: Image of the boundary of $[-1/2,1/2]^2$ under the generalized Gauss map.}
\label{fig:squareCF}
\end{figure}

\subsubsection{Big-diamond CFs}
\label{subsubsec:BigDiamondCF}
In view of Lemma \ref{lemma:EuclideanExpandingRegion}, one can consider a CF in $\R^{1,1}$ with inversion $\iota(x)=\overline x/Q(x)$, fundamental domain the big diamond $K=\{(x_1,x_2): |x_1|+|x_2|\leq 1\}$ and digits in the lattice $\{a\frac{(1,1)}{2}+b\frac{(1,-1)}{2} : a,b \in 2\Z\}$. In view of Lemma \ref{lemma:changeOfCoordinates}, these are exactly a product of two copies a CFs in $\R$ with inversion $\iota(x)=1/x$, lattice $2\Z$, and fundamental domain $[-1,1]$. The system is convergent (see e.g.~\cite{lukyanenko2022convergence} for a quick argument). A closely related folded variant with $\iota(x)=|1/x|$ was studied by Schweiger \cite{schweiger1982continued}, who showed that it was ergodic with an infinite invariant measure, but did not satisfy Renyi's condition due to an indifferent fixed point. It is not clear whether the one-dimensional even CFs are exact (in a sense appropriate for infinite measures). Likewise, it is unclear whether the product equivalent to big-diamond CFs is ergodic.

\subsection{Three-dimensional system}
\label{subsec:3DSystem}
For $2+1$-dimensional Lorentzian continued fractions, 
we observe that the quadratic form $-\det$ on the vector space of real $2\times 2$ symmetric matrices $\mathrm{Sym}_2(\bR)$ has signature $(2,1)$, making it into a Lorentzian vector space. The basis
\[E_1=\begin{pmatrix} 1 & 0\\0&-1\end{pmatrix}, E_2=\begin{pmatrix} 0 & 1\\1 & 0\end{pmatrix}, E_3=\begin{pmatrix} 1 & 0\\0 & 1\end{pmatrix}\]
is orthogonal for the associated bilinear form, so the map
\[(x_1,x_2,x_3) \mapsto \begin{pmatrix}
x_1+x_3 & x_2\\x_2 & -x_1+x_3\end{pmatrix}\]
is an isomorphism of Lorentzian vector spaces, with inverse
\[\begin{pmatrix}
a & b\\b & c\end{pmatrix} \mapsto \left(\frac{a-c}{2}, b, \frac{a+c}{2}\right).\]

\begin{figure}[b]
\includegraphics[width=.5\textwidth]{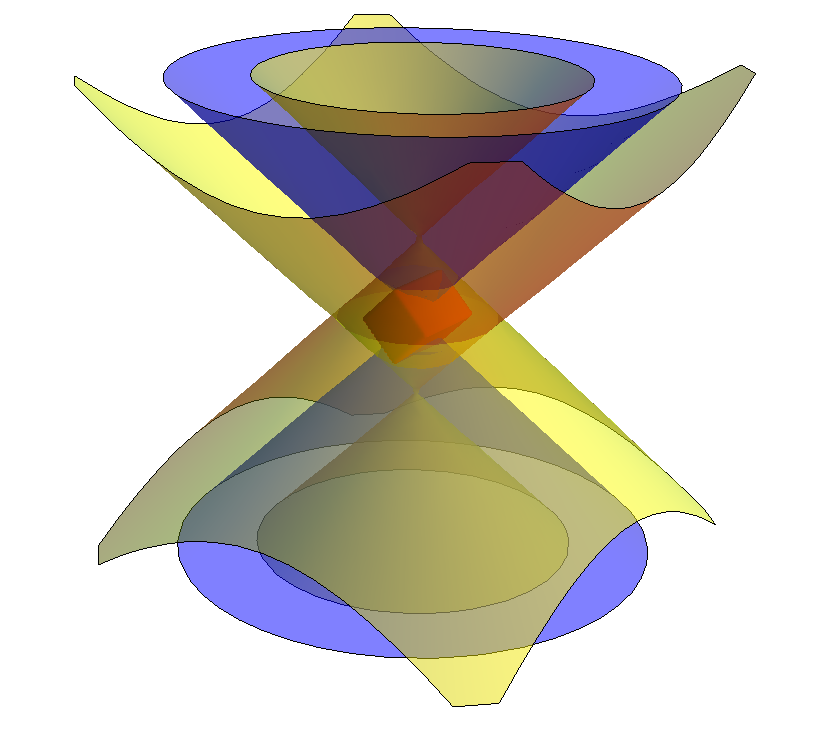}
\caption{The fundamental domain in $\R^{2,1}$ of the Lorentzian continued fractions is shown in red. The light cone $\mathcal N$ at the origin is shown in blue, while the light cones at $(0,0,\pm 1)$ are shown in yellow. The region between these two light cones is the Euclidean expanding region for the inversion.}
\label{fig:3D}
\end{figure}

In this basis, the Lorentzian inversion satisfies $\iota(x_1,x_2,x_3)=\frac{(x_1,x_2,-x_3)}{x_1^2 + x_2^2 -x_3^2}$:
\[\iota\begin{pmatrix}
x_1+x_3 & x_2\\x_2 & -x_1+x_3\end{pmatrix} = \frac{1}{x_1^2+x_2^2-x_3^2}\begin{pmatrix}
x_1-x_3 & x_2\\x_2 & -x_1-x_3\end{pmatrix},\]
and so by defining $\overline{x_1 E_1 + x_2 E_2 + x_3 E_3} = x_1 E_1 + x_2 E_2 - x_3 E_3$ we can write $\iota(M) = \overline{M}^{-1}.$ We use $\mathcal{Z} = \Sym^2(\bZ) \subset \Sym^2(\bR)$ as our lattice of translations, and

\[K=\left\{\left.\begin{pmatrix}a & b\\b & c \end{pmatrix} ~\right|~ a,b,c \in [-1/2,1/2]\right\}\]
as our fundamental domain.

In coordinates $(x_1,x_2,x_3)$ as above, this lattice is generated by $(1/2,0,1/2)$, $(0,1,0)$, and $(-1/2,0,1/2)$. The fundamental domain is defined by the inequalities $|a-c|\le 1$, $|a+c|\le 1$, and $|b|\le \frac12$. This region is in the expanding region for $\iota$ given by Lemma \ref{lemma:EuclideanExpandingRegion}, as shown in Figure \ref{fig:3D}.

Experimental evidence suggests this CF is convergent and ergodic with respect to a finite invariant measure.

\bibliographystyle{abbrv}  
\bibliography{Bibliography}
\end{document}